\documentclass[11pt]{amsart}

\usepackage{amsmath,amsthm,amsfonts,amssymb,latexsym,mathtools}		
\usepackage{array}				
\usepackage{enumerate}			

\usepackage{graphicx,epsfig,color,mathrsfs}
\graphicspath{{Figures/}}			
\usepackage{minibox} 			
\usepackage{subfigure} 			
\usepackage{epsf}				
\usepackage{color}				

\usepackage{hyperref}		
\usepackage{todonotes}              	
\usepackage[margin=1.3in]{geometry} 

\newtheorem{theorem}{Theorem}[section]
\newtheorem{lemma}[theorem]{Lemma}

\newtheorem{proposition}[theorem]{Proposition}
\newtheorem{corollary}[theorem]{Corollary}

\newtheorem{observation}[theorem]{Observation}

\newtheorem*{main:CF}{Theorem \ref{thm:CF}}

\theoremstyle{definition}
\newtheorem{definition}[theorem]{Definition}

\newtheorem{remark}[theorem]{Remark}

\newcommand {\R}{\mathbb{R}}
\newcommand {\N}{\mathbb{N}}

\newcommand{\CF}{\mathcal{C}F}
\newcommand{\DF}{\mathcal{D}F}

\newcommand{\CPB}{\mathcal{C}PB}

\newcommand{\Poset}{\mathcal{P}}

\DeclareMathOperator{\CAT}{CAT}
\DeclareMathOperator{\Diag}{Diag}

\newcommand{\defeq}{\mathrel{\vcentcolon =}}

\begin{document}

\title{On Belk's classifying space for Thompson's group $F$}

\author[L.~Sabalka]{Lucas~Sabalka}
      \address{Lincoln, NE}
\email{sabalka@gmail.com}

\author[M.~C.~B.~Zaremsky]{Matthew~C.~B.~Zaremsky}
      \address{Department of Mathematical Sciences\\
               Binghamton University\\
               Binghamton, NY 13902
}
\email{zaremsky@math.binghamton.edu}

\date{\today}
\thanks{The second author was supported by the SFB~701 in Bielefeld and SFB~878 in M\"unster during the course of this work}

\begin{abstract}
 The space of configurations of $n$ ordered points in the plane serves as a classifying space for the pure braid group $PB_n$. Elements of Thompson's group $F$ admit a model similar to braids, except instead of braiding the strands split and merge. In Belk's thesis, a space $\CF$ was considered, of configurations of points on the real line allowing for splitting and merging, and a proof was sketched that $\CF$ is a classifying space for $F$. The idea there was to build the universal cover and construct an explicit contraction to a point. Here we start with an established $\CAT(0)$ cube complex $X$ on which $F$ acts freely, and construct an explicit homotopy equivalence between $X/F$ and $\CF$, proving that $\CF$ is indeed a $K(F,1)$.
\end{abstract}

\maketitle

\section{Introduction}\label{sec:intro}

Thompson's group $F$ is the group of all piecewise linear homeomorphisms of $[0,1]$ with finitely many breakpoints, all at dyadic rational numbers, and with slopes all powers of two. This group is a rare example of a finitely presented, torsion-free group with infinite cohomological dimension. Elements of $F$ can also be represented by \emph{strand diagrams}, which are like braid diagrams except the strands \emph{split} and \emph{merge} rather than braiding; see \cite{belk14}.

We say that a space $X$ is a \emph{classifying space} for a group $G$ if $\pi_1(X) \cong G$ and the universal cover $\widetilde{X}$ of $X$ is contractible.  This is equivalent to saying $X$ is a $K(G,1)$ space, meaning that $\pi_k(X)$ is trivial when $k\neq 1$ and is $G$ when $k=1$. This condition that the higher homotopy groups vanish is called being \emph{aspherical}. Classifying spaces are unique up to homotopy equivalence, and one often seeks to find ``nice'' representatives in the homotopy type for a given $K(G,1)$.

It is a classical fact that there is a classifying space for the pure braid group $PB_n$ given by the space $\CPB_n$ of ordered configurations of $n$ points in the plane. If the points are unordered, we find a classifying space for the braid group $B_n$. Since strand diagrams for Thompson's group $F$ resemble braid diagrams, one might expect there to be a classifying space for $F$ given by configurations of some sort. In Belk's thesis \cite{belk04}, a certain space $\CF$ of configurations of points on the real line was considered, and a proof was sketched that $\CF$ is a classifying space for $F$. We prove Belk's claim that $\CF$ is indeed a classifying space for $F$.

\begin{main:CF}
 The space $\CF$ is a classifying space for $F$.
\end{main:CF}

We do not directly follow Belk's proof sketch, but rather make use of a certain space $X$, discussed in Section~\ref{sec:stein}. This $X$ is a $\CAT(0)$ cube complex on which $F$ acts freely; it was first considered by Stein \cite{stein92}, and has also appeared in work of Brown \cite{brown92} and Farley \cite{farley03}. We exhibit an explicit homotopy equivalence between $X/F$ and $\CF$, which proves that $\CF$ is indeed a $K(F,1)$.

\subsection*{Acknowledgments}

The first author thanks Keith Jones for the discussions that led to this project. The second author thanks Kai-Uwe Bux and Stefan Witzel for helpful discussions and suggestions. We are also grateful to Jim Belk and Ross Geoghegan for helpful conversations, and to an anonymous referee for catching a critical error in a previous version of this paper.


\section{Spaces of configurations}\label{sec:config_spaces}

In this section we recall the spaces of configurations in the plane serving as classifying spaces for pure braid groups and braid groups, and then define Belk's space $\CF$ of certain configurations in the line.

\subsection{Braid group examples}\label{sec:braid}

There are some well known classifying spaces for the braid group $B_n$ and pure braid group $PB_n$ realized as \emph{configuration spaces} of points in the plane \cite{birman05}. Let $\CPB_n$ denote the space of all $n$-tuples of pairwise distinct complex numbers, i.e.,
	$$\CPB_n \defeq \mathbb{C}^n \setminus \Diag_n \text{,}$$
where $\Diag_n \defeq \{(z_1, \dots, z_n) \mid z_i = z_j \text{ for some } i\neq j\}$ is the fat diagonal in $\mathbb{C}^n$.

\begin{theorem}\label{thm:K(PB_n,1)}
 The space $\CPB_n$ is a classifying space for $PB_n$.
\end{theorem}

If we quotient out the action of the symmetric group $\Sigma_n$ on the coordinates in $\mathbb{C}^n$, yielding configurations of $n$ \emph{un}ordered points in the plane, then we obtain a classifying space for the (usual, non-pure) braid group $B_n$.

\subsection{Belk's space}\label{sec:classifying_space_F}

We now define the space $\CF$ constucted in Belk's thesis \cite{belk04}. Let $\CF_n$ denote the space of all $n$-tuples of real numbers, $(t_1, \dots, t_n)$, such that:
\begin{enumerate}
 \item the entries are non-decreasing, i.e., $t_1 \leq \dots \leq t_n$, and
 \item for all $i = 1, \dots, n-2$, $t_{i+2}-t_i \geq 1$.
\end{enumerate}
This second condition should be thought of as saying that no three distinct entries are too close together.  For example, $(1,1,2)$ is a point in $\CF_3$, but $(1,1,3/2)$ is not. Let
$$\CF \defeq \left(\coprod\limits_{n=1}^\infty \CF_n\right)/\sim$$
denote the disjoint union of the spaces $\CF_n$, subject to the identifications
	$$(t_1, \dots, t_i, \dots, t_n) \sim (t_1, \dots, t_i, t_i, \dots, t_n)$$
whenever $(t_1, \dots, t_i, t_i, \dots, t_n) \in \CF_{n+1}$, that is whenever $t_i$ is at least~$1$ away from its neighbors.

Notice that for each $n$, $\CF_n$ is contractible, for instance to the point $p_n \defeq (1, \dots, n)$ in $\CF_n$. The contraction is given by the homotopy that at time $0 \leq t \leq 1$ takes the point $\vec{t} \defeq (t_1, \dots, t_n)$ to $p_nt + (1-t)\vec{t}$.  Heuristically, the fundamental group of $\CF$ (which will be $F$) will come from the identifications arising from $\sim$.  Points in $\CF_n$ are identified with points in $\CF_{n+1}$ in at most~$n$ ways, once for each allowed bifurcation of an entry in $\CF_n$, so non-trivial elements of $\pi_1(\CF)$ can arise, for example, by splitting one entry and then merging into a different entry.  Then non-trivial relations in $\pi_1(\CF)$ arise from the fact that splits and merges that are far enough apart may happen in any order.

\medskip

The rest of the paper is devoted to proving Belk's claim:

\begin{theorem}\label{thm:CF}
 The space $\CF$ is a classifying space for $F$.
\end{theorem}

\section{Stein's space}\label{sec:stein}

Rather than directly calculating $\pi_1(\CF)$ and showing that $\CF$ is aspherical, we will take a known classifying space of $F$, which will be denoted $\overline{X}$ below, and map it into $\CF$. We will show that this map is a homeomorphic embedding, and then show that $\CF$ homotopes to the image of this embedding. The space $\overline{X}$ will be the quotient $X/F$, where $X$ is a $\CAT(0)$ cube complex on which $F$ acts freely. This space was first introduced by Stein \cite{stein92}, and subsequently used by Brown \cite{brown92}, among others; it was shown to be $\CAT(0)$ by Farley \cite{farley03}. We will call $X$ \emph{Stein's space}, since it was first built by her, though one could also call it the Stein--Brown--Farley complex.

\noindent \textbf{Strand diagrams.} To describe $X$, it is convenient to use the \emph{strand diagram} model for elements of $F$, as in \cite{belk14}.  For our purposes, a strand diagram $\Delta$ is a finite directed graph embedded in the infinite ``strip'' $[0,\infty) \times [0,k]$ for some positive number $k$ satisfying the following properties:

\begin{enumerate}
 \item $\Delta \cap ([0,\infty) \times \{0\})$ is of the form $\{1,\dots,m\} \times \{0\}$ for some $m \in \N$. These points are degree $1$ or $2$ vertices of $\Delta$, called \emph{sources}.
 \item $\Delta \cap ([0,\infty) \times \{k\})$ is of the form $\{1,\dots,n\} \times \{k\}$ for some $n \in \N$. These points are degree $1$ or $2$ vertices of $\Delta$, called \emph{sinks}.
 \item All edges are oriented from top to bottom, with no horizontal edges, so that every maximal directed edge path flows from a source downward to a sink.
 \item Vertices of degree $2$ that are not sources or sinks must have one incoming edge and one outgoing edge.
 \item Every vertex that has degree greater than $2$ must be either trivalent, with at least one incoming edge and at least one outgoing edge, or have degree $4$, with exactly two incoming edges and two outgoing edges.
\end{enumerate}

We will picture the sources on top and the sinks on the bottom, so the interval $[0,k]$ is oriented with $0$ on top and $k$ on the bottom.

A vertex of $\Delta$ with two outgoing edges is called a \emph{split (vertex)}, and a vertex with two incoming edges is called a \emph{merge (vertex)}.  A vertex of degree $4$ is both, so we call such a vertex a \emph{merge-split}\footnote{Other texts, e.g., \cite{belk14}, disallow merge-splits.  We allow them here to permit elementary multiplication of a merge caret by a split caret, defined later.  Note though that merge-splits are unnecessary in general:  all merge-splits may be removed via reduction.}.  A connected subgraph consisting of a split vertex and its two downward edges is a \emph{split caret}, while a connected subgraph consisting of a merge vertex and its two upward edges is a \emph{merge caret}. Collectively, these are called \emph{carets}.  If there are $m$ sources and $n$ sinks, we call the strand diagram an $(m,n)$-\emph{strand diagram}. See Figure~\ref{fig:strandsa} for an example.

\begin{figure}[ht]
\includegraphics[width=4in]{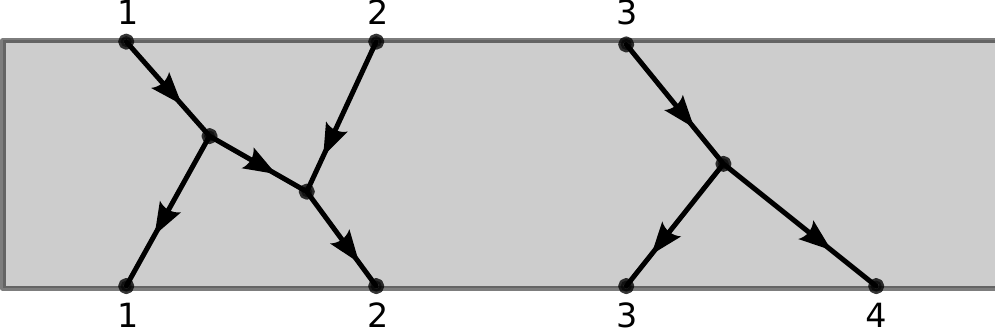}

\caption{A $(3,4)$-strand diagram.}\label{fig:strandsa}
\end{figure}

We will actually consider equivalence classes of strand diagrams, with equivalence relation $\sim_s$ given by the transitive closure of the following moves:

\begin{itemize}
 \item A homeomorphism of the strip $[0,\infty) \times [0,k]$ 
 \item Deletion or insertion of degree $2$ vertices on edges
 \item Vertical rescaling of the strip, i.e., varying $k$
 \item \emph{Reduction} and \emph{expansion} of the splits and merges
 \item \emph{Diagram reduction} and \emph{diagram expansion} of the sources and sinks.
\end{itemize}

A \emph{reduction} comes in two forms. First, if there is a merge whose outgoing edge is the incoming edge of a split, or the merge is a merge-split (so the ``edge'' between the merge and the split has length $0$), then we can replace this subgraph with two parallel edges.  Second, if there is a split whose two outgoing edges form the incoming edges of a merge, then we can replace this subgraph with a single edge. See Figure~\ref{fig:strandsb}, and also \cite[Figure~3]{belk14}. An \emph{expansion} is just the reverse procedure of reduction.

\begin{figure}[ht]
\includegraphics[width=2in]{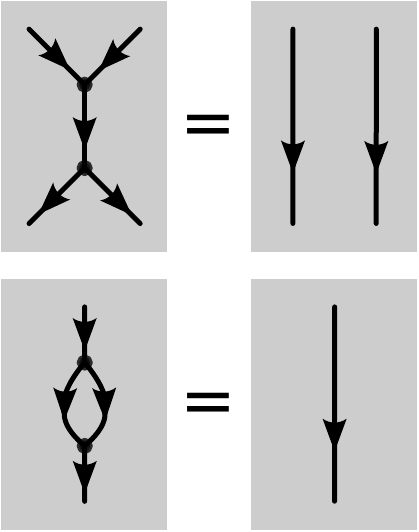}

\caption{From left to right, these pictures show the two forms of reduction of strand diagrams; from right to left, they show the two forms of expansion.}\label{fig:strandsb}
\end{figure}

A \emph{diagram reduction} of a source or sink in a strand diagram is defined as follows. Suppose we have a subgraph consisting of a split vertex and its three edges, the incoming one of which has a degree $1$ source $s$ as its other endpoint. Call the endpoints of the other two edges $v$ and $w$ (up to reduction we may assume $v\neq w$). Then a diagram reduction of this source $s$ amounts to replacing this subgraph (a tripod) with a split caret whose split vertex is now $s$ and whose other two vertices are still $v$ and $w$. This really just amounts to collapsing the edge from the source to the split vertex. Similarly a diagram reduction of a sink amounts to collapsing an edge connecting a merge vertex to a degree $1$ sink. Diagram expansion is the reverse of diagram reduction. As a remark, we will sometimes refer to diagram reduction/expansion as reduction/expansion, when the distinction is not important.

\medskip

The equivalence classes of strand diagrams under $\sim_s$ form a groupoid: we (right) multiply an $(m,n)$-strand diagram class by an $(n,p)$-strand diagram class by picking representatives and stacking the former on top of the latter, after which we obtain an $(m,p)$-strand diagram, as in Figure~\ref{fig:strand_mult}. We use $\ast$ to denote multiplication of strand diagrams up to equivalence. Thanks to the nature of $\sim_s$, inverse elements under $\ast$ are obtained simply by reflection about the ray $[0,\infty) \times \{k/2\}$. In particular the $(1,1)$-strand diagram classes form a group, which is exactly Thompson's group $F$. From this point on we will use the term ``strand diagram'' to mean an equivalence class under $\sim_s$.

\begin{figure}[ht]
\includegraphics[width=5in]{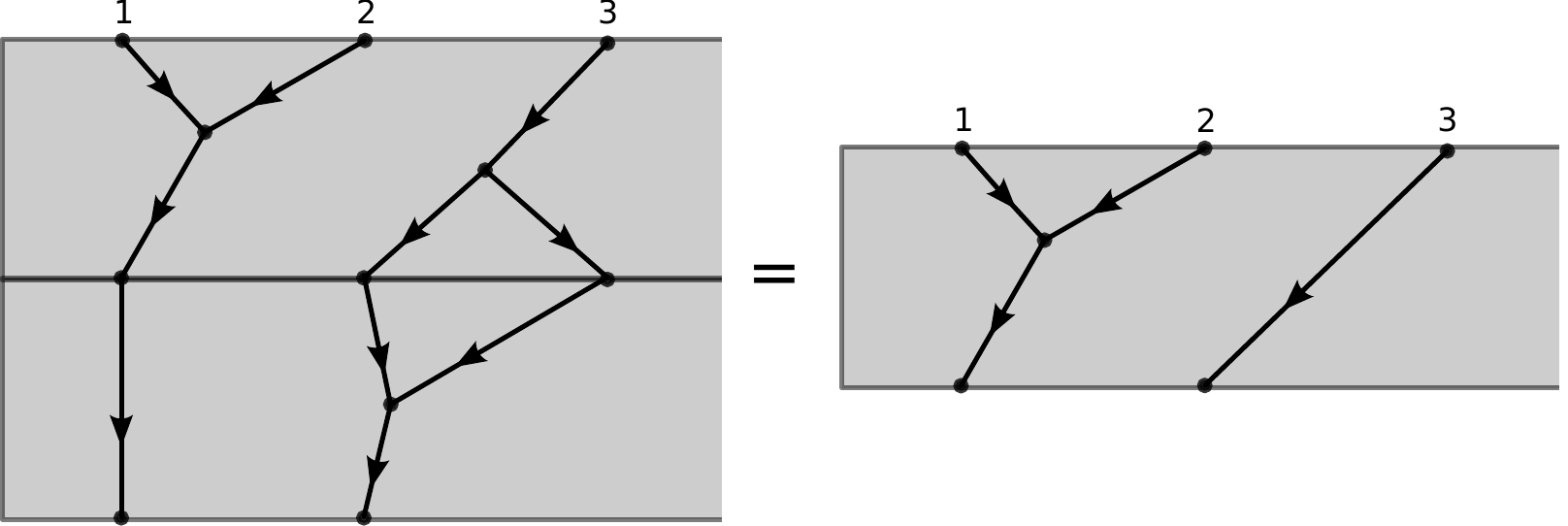}

\caption{A $(3,3)$-strand diagram times a $(3,2)$-strand diagram equals a $(3,2)$-strand diagram.}
\label{fig:strand_mult}
\end{figure}

\medskip

\noindent \textbf{Splitting and merging.} There are certain multiplications in the groupoid of all strand diagrams that are particularly important. If (some representative of an equivalence class of) a strand diagram $\Phi$ has no undirected loops we call it a \emph{forest}.  If $\Phi$ contains no merge vertices, and if every split vertex is a source, then it is called an \emph{elementary splitting forest}, and right multiplication by $\Phi$ is called an \emph{elementary splitting}.  The idea is that for a strand diagram $\Delta$, under elementary splitting by $\Phi$ each strand at the bottom of $\Delta$ could either continue downward unchanged, or it could split into two strands in $\Delta \ast \Phi$.  The inverse of an elementary splitting forest is called an \emph{elementary merging forest}, and right multiplication by such a forest is called \emph{elementary merging}.

If a strand diagram $\Delta'$ can be obtained from a strand diagram $\Delta$ by a sequence of elementary splittings then $\Delta'$ is \emph{obtained from $\Delta$ by splitting}.  For each $n\in \N$ consider the set $\Poset_{1,n}$ of $(1,n)$-strand diagrams, and let
      $$\Poset_1 \defeq \bigcup\limits_{n\in\N} \Poset_{1,n}\text{.}$$
This set has a partial ordering $\le$, given by $x\le y$ if $y$ is obtained from $x$ by splitting.

\begin{observation}\label{obs:poset_contractible}
 Any two elements of $\Poset_1$ have an upper bound, and hence $|\Poset_1|$ is contractible.
\end{observation}

\begin{proof}
 Let $x,y\in \Poset_1$. There exist $x'\ge x$ and $y'\ge y$ such that $x'$ and $y'$ have no merges. Then it is clear that $x'$ and $y'$ have an upper bound. The partially ordered set $\Poset_1$ is therefore directed, and so has contractible geometric realization.
\end{proof}

Since $F$ is the group of $(1,1)$-strand diagrams, there is an action of $F$ on $\Poset_1$ given by left multiplication, which extends to an action on the geometric realization $|\Poset_1|$. This action is free, and for each $n$ it is transitive on the vertex subset $\Poset_{1,n}$. Since $|\Poset_1|$ is contractible and the $F$-action is free, we immediately see that:
\begin{quote}\centering
 $|\Poset_1|/F$ is a classifying space for $F$.
\end{quote}
This is not quite the space that we want though.

\medskip

\noindent \textbf{Elementary simplices.} We wish to ``throw away'' many of the simplices in $|\Poset_1|$, i.e., chains in $\Poset_1$. Given a chain $x_0<x_1<\cdots <x_k$, call the chain \emph{elementary} if we obtain $x_k$ from $x_0$ by an elementary splitting. This condition is closed under taking subchains, and so the subspace $X$ of $|\Poset_1|$ consisting of the elementary simplices is a subcomplex. This subcomplex is clearly $F$-invariant.

The simplices in $X$ can be glommed together, giving $X$ the structure of a metric cube complex. For any $x<y$ an elementary chain, the chains of the form $x_0<\cdots<x_n$ with $x=x_0$ and $y=x_n$ form the simplices of an $n$-dimensional cube. We call $x$ the \emph{top} vertex of the cube and $y$ the \emph{bottom} vertex. The metric is realized by identifying each $n$-cube with the unit cube $[0,1]^n$. It may seem odd to call $x$ the top and $y$ the bottom, when $x<y$, but since elementary splitting happens at the bottom of the diagram, this terminology will be more in sync with pictures that occur later.

\begin{lemma}[\cite{brown92,stein92,farley03}]\label{lem:stein_space_contractible}
 $X$ is contractible.
\end{lemma}

This has been proved in \cite{brown92,stein92,farley03}, and more recently using the language of strand diagrams in \cite{bux14}, so we will not repeat the proof here. For the interested reader, $X$ is not just contractible but is in fact a $\CAT(0)$ cube complex, as proved by Farley.

\begin{corollary}\label{cor:stein_classifying}
 $\overline{X} \defeq X/F$ is a classifying space for $F$. \qed
\end{corollary}

In Section~\ref{sec:stein_to_belk} we will embed $\overline{X}$ into $\CF$ as a subspace. To do this, we first discuss a parameterization of $X$, by \emph{weighted strand diagrams}, which will allow us to describe the map $\overline{X}\hookrightarrow \CF$ very explicitly.

\medskip

\subsection{Generalized strand diagrams}\label{sec:gen_strand_diagrams}

We begin by defining \emph{weighted elementary forests}.  If a forest $\Phi$ has the property that every connected component is either an edge or a caret (either with a single split/source or with a single merge/sink), we call $\Phi$ \emph{elementary}. Elementary splitting forests and elementary merging forests are both examples of elementary forests, but elementary forests may contain both split carets and merge carets.  The idea is that when we take a strand diagram and right multiply by an elementary forest, each strand at the bottom of the diagram could do one of three things:
	\begin{enumerate}
	 \item It could continue downward unchanged,
	 \item It could split into two strands (if a split caret is attached to it), or
	 \item It could merge with a neighboring strand (if a merge caret is attached to it).
	\end{enumerate}
Right multiplying by an elementary forest is called an \emph{elementary multiplication}.

An elementary forest can always be decomposed into the product of an elementary splitting forest and an elementary merging forest, or into the product of an elementary merging forest and an elementary splitting forest. Given an elementary forest $\Phi$, let $\Phi_{split}$ be the forest obtained from $\Phi$ by replacing each merge caret in $\Phi$ with two parallel lines (hence ``undoing'' the merges of $\Phi$), and shifting the sinks as necessary. Similarly let $\Phi_{merge}$ be obtained by replacing each split caret with one edge (so undoing the splits of $\Phi$), and shifting the sinks as necessary. Then right multiplication by $\Phi$ is the same as right multiplication first by $\Phi_{split}$ and then by some elementary merging forest, and is also the same as right multiplication first by $\Phi_{merge}$ and then by some elementary splitting forest. (Note that we do not claim $\Phi=\Phi_{split} \ast \Phi_{merge}$; by considering the number of sources and sinks, this product is usually not even defined.)

Observe now that if $v$ is a vertex of $X$ and $\Phi$ is an elementary forest such that $v \ast \Phi$ exists (meaning that the number of sinks of $v$ equals the number of sources of $\Phi$), the vertices $v$ and $v \ast \Phi$ are opposite vertices of a cube in $X$, namely the cube with bottom vertex $v \ast \Phi_{split}$ and top vertex $v \ast \Phi_{merge}$.  The dimension of this cube equals the number of carets in $\Phi$.  It follows that an equivalent definition of $X$ is as the complex of all cubes corresponding to any elementary multiplication, not just elementary splittings (though many such cubes coincide).

Given an elementary forest $\Phi$, a \emph{weighting} on $\Phi$ is a map from the set of carets of $\Phi$ to the interval $[0,1]$.  We introduce an equivalence relation $\sim_f$ among weighted elementary forests corresponding to adding and deleting weight zero components.  Specifically, a given weighted elementary forest is \emph{equivalent} to the weighted elementary forest that is the result of replacing each split caret having weight $0$ with a single edge and each merge caret having weight $0$ with two parallel edges, and then shifting sink vertices appropriately so that the set of sinks remains a set of consecutive natural numbers starting at $1$.  Equivalence $\sim_f$ is then the transitive closure of this identification.

\begin{definition}[Generalized strand diagrams]\label{def:gen_strand}
 A \emph{generalized strand diagram} is a $(1,n)$-strand diagram for $n\in\N$ (i.e., a vertex of $X$), together with a single multiplication by a weighted elementary forest.
\end{definition}

Heuristically, the weight of each caret in a generalized strand diagram should be thought of as the percent of ``progress'' toward fully attaching that caret in the multiplication.  In particular, a usual elementary multiplication can be thought of as a weighted elementary multiplication in which all carets have weight $1$.  Note also that the equivalence relation $\sim_s$ for the usual strand diagrams can be extended to generalized strand diagrams.  The only moves that require explanation are reduction and expansion.  If a split (with weight $1$) is multiplied by a merge with weight $w$, that is equivalent to replacing the split-merge pair with a single split with weight $1-w$.  Similarly, if a merge is multiplied by a split with weight $w$, that is equivalent to replacing the merge-split pair with a single merge with weight $1-w$.  This is how reduction works, and expansion is again just the reverse.  The relations $\sim_s$ and $\sim_f$ together induce an equivalence relation $\sim_g$ on generalized strand diagrams. It is not difficult to see that any generalized strand diagram is equivalent under $\sim_g$ to a \emph{unique} generalized strand diagram to which no reduction move applies and for which no carets have weight $0$. Note that ``no carets have weight $0$'' may be vacuously satisfied, if there are no carets at all.

The vertices of $X$ are (equivalence classes of) generalized strand diagrams where the associated weighted elementary forests have weight $1$ on each caret.  We now want to extend this parameterization to all of $X$.  Let $K$ be a cube of dimension $n$ and $x$ a vertex of $K$, with $y$ the vertex opposite $x$ in $K$.  Say $y$ is obtained from $x$ via elementary multiplication by the elementary forest $\Phi$, which necessarily has $n$ carets.  By weighting $\Phi$ with different caret weightings, we get the coordinates of a cube that can be identified with $K$.  More precisely, if we identify $K$ with $[0,1]^n$, and if the weighting assigns to the $i^{\text{th}}$ caret of $\Phi$ the weight $w_i$ (with $0\leq w_i\leq 1$), then the resulting point $q$ in $K$ is given by the coordinates $(w_1,\dots,w_n)\in [0,1]^n$.  In this way, for a fixed initial vertex $x$ in $K$, we can assign to $q$ a unique generalized strand diagram $x \ast \Phi_q$, where $\Phi_q$ is $\Phi$ with the appropriate weightings on its carets. We call $x \ast \Phi_q$ the \emph{parameterization of $q$ in $K$ relative $x$}.

Note that for a given point $q$ in $X$, different choices of cubes $K$ containing $q$ and different choices of initial vertex $x$ in $K$ may \emph{a priori} yield different parameterizations, since $x$ always corresponds to the point $(0,\dots,0)$ and $y$ always corresponds to the point $(1,\dots, 1)$. However, we now claim that the parameterization by generalized strand diagrams is in fact unique, i.e., independent of $x$ and $K$, and varies continuously.

\begin{lemma}\label{lem:param_cts}
 The parameterization of $X$ by generalized strand diagrams is independent of choices of $x$ and $K$, and varies continuously over $X$.
\end{lemma}

\begin{proof}
 For a given initial vertex $x$, the parameterization is continuous on closed cubes containing $x$ by definition.  Hence we need only show that the parameterization of a given $n$-cube $K$ is independent of the choice of initial vertex $x$, and then that these parameterizations agree on intersections of cubes.
 
 Let $x$ and $x'$ be two different choices of initial vertex in $K$ with associated elementary forests $\Phi$ and $\Phi'$, so $x \ast \Phi$ (respectively $x' \ast \Phi'$) is the opposite vertex to $x$ (respectively $x'$) in $K$.  For any point $q$ in $K$, let $(w_1(q), \dots, w_n(q))$ be the coordinates of $q$ in $K$ relative $x$, with $w'_i(q)$ defined similarly for $x'$. Let $\Phi_q$ be the weighted elementary forest obtained from $\Phi$ by assigning the weights $w_1(q), \dots, w_n(q)$ to the carets of $\Phi$. Define $\Phi'_q$ similarly. Hence $q$ is parameterized by $x \ast \Phi_q$ relative $x$, and by $x' \ast \Phi'_q$ relative $x'$.  Note that $x \ast \Phi_{x'}$ is just $x'$ as a (usual) strand diagram.  Similarly $x' \ast \Phi'_x = x$.  Hence we have
$$x = x \ast \Phi_{x'} \ast \Phi'_x \text{,}$$
so $\Phi_{x'}=(\Phi'_x)^{-1}$ in the groupoid of all strand diagrams.  This tells us that for any $q$,
$$\Phi_{x'} \ast \Phi'_q = (\Phi'_x)^{-1} \ast \Phi'_q$$
is a weighted elementary forest.

 Now we have $x' \ast \Phi'_q = x \ast \Phi_{x'} \ast \Phi'_q$, and since $\Phi_{x'} \ast \Phi'_q$ is a weighted elementary forest,  $x \ast \Phi_{x'} \ast \Phi'_q$ is a generalized strand diagram that parameterizes $q$ relative $x$. By uniqueness,
 $$x \ast \Phi_{x'} \ast \Phi'_q \sim_g x \ast \Phi_q \text{,}$$
 and we conclude that $x' \ast \Phi'_q \sim_g x \ast \Phi_q$.
 
 Finally, we need to show that the parameterizations agree on intersections of cubes.  Let $L$ be a face of $K$ and $x$ a vertex in $L$. Let $\Phi$ and $\Psi$ be such that the vertex $y \defeq x \ast \Phi$ is opposite $x$ in $K$ and the vertex $z \defeq x \ast \Psi$ is opposite $x$ in $L$.  For a point $q$ in $L$, it now suffices to show that $x \ast \Phi_q \sim_g x \ast \Psi_q$. But this is clear because $\Phi_q \sim_f \Psi_q$.
\end{proof}

The reader may find it helpful to imagine how generalized strand diagrams change while traveling through $X$. Moving in $X$ amounts to multiplication by weighted elementary forests (which includes elementary splitting and elementary merging).  Multiplying by a split caret that has weight $w$ should be thought of as splitting the associated strand into two strands, but where the two new strands have only moved distance $w$ apart (on their way to distance $1$), yielding a ``short'' caret.  Similarly, multiplying by a merge caret that has weight $w$ should be thought of as merging the corresponding two strands together, but not completely, only moving them a total distance of $w$ closer together.  When the weight on a caret is $0$, nothing is done, while if the weight on a caret is $1$, the corresponding split or merge is complete.

\section{From $\overline{X}$ to $\CF$}\label{sec:stein_to_belk}

We now wish to embed $\overline{X}$ into $\CF$ as a subspace $\DF$ and then homotope $\CF$ to $\DF$. Let $q$ be a point in $X$. Choose a cube $K$ containing $q$ and a vertex $x$ contained in $K$. Then $q$ is a generalized strand diagram obtained from $x$ by multiplication by a weighted elementary forest $\Phi$.  Number the connected components of $\Phi$ by $c_1, \dots, c_\ell$.  Then each $c_i$ is either a split caret, a merge caret, or an edge.  Each caret of $\Phi$ is given a weight as determined by the location of $q$ within $K$, say the caret $c_i$ has weight $w_i$.

We are now ready to define the configurations map $C \colon X \to \CF$.  It is given by
$$C(q) \defeq (L_1,R_1,L_2,R_2, \dots, L_\ell, R_\ell) \text{,}$$
where the $L_i$ and the $R_i$ are defined to be:
	$$L_i \defeq i + \sum\limits_{\stackrel{j = 1}{\text{$c_j$ a split}}}^{i-1} w_j + \sum\limits_{\stackrel{j = 1}{\text{$c_j$ a merge\textcolor{white}{l}}}}^{i-1} (1-w_j)$$
and
	$$R_i \defeq i + \sum\limits_{\stackrel{j = 1}{\text{$c_j$ a split}}}^{i} w_j + \sum\limits_{\stackrel{j = 1}{\text{$c_j$ a merge\textcolor{white}{l}}}}^{i} (1-w_j) \text{.}$$

Note that $L_1=1$. Also note that differences of consecutive entries are given by: $L_{i+1}-R_i=1$ for all $i$, and $R_i-L_i$ is either $1$, $w_i$ or $1-w_i$.

\begin{figure}[ht]
\centering
\hspace{-.3in}
$\vcenter{\hbox{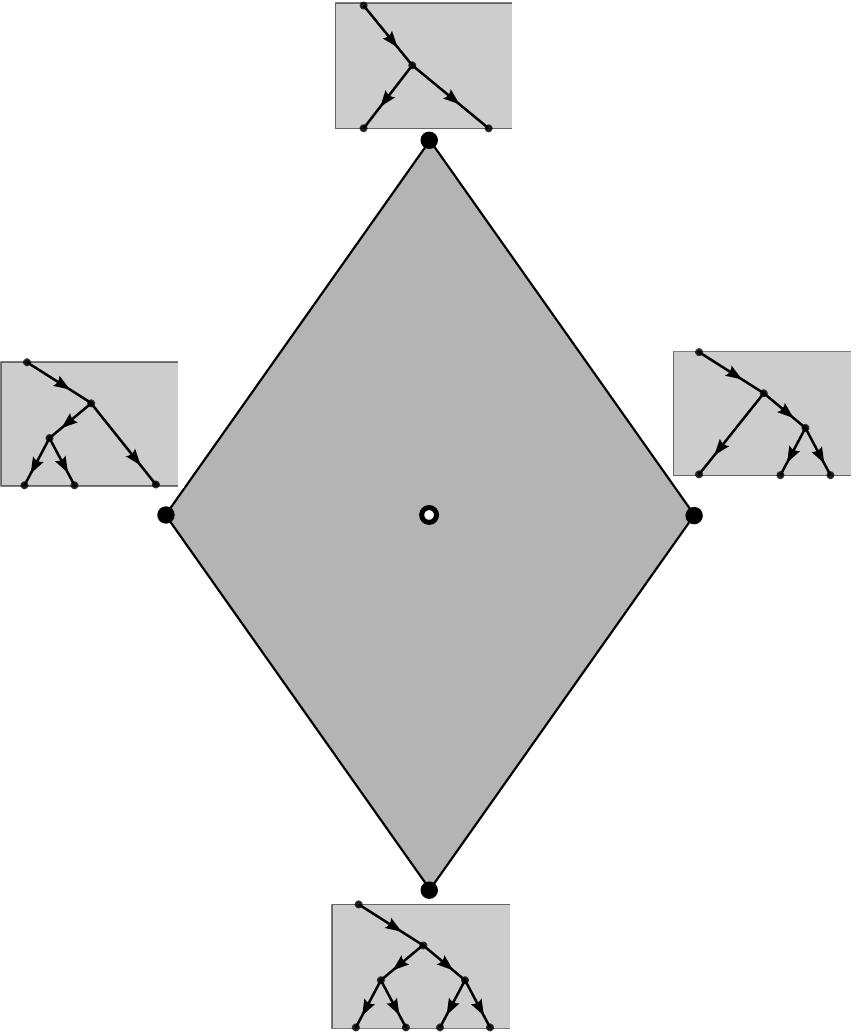}}$
\qquad
$\vcenter{\hbox{$\buildrel{C}\over{\longrightarrow}$}}$
\qquad
$\vcenter{\hbox{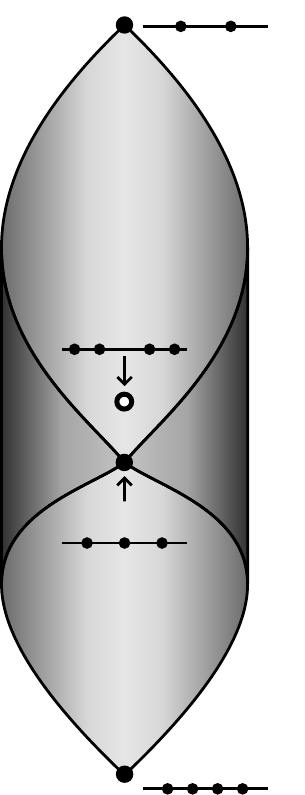}}$
\hspace{.6in}
\caption{The configurations map. Here, $x[i,j]$ denotes the generalized strand diagram obtained by multiplying the top vertex $x$ (labeled $x[0,0]$) by the weighted elementary forest $\Phi$ that has exactly two split carets, with weight $i$ and $j$, respectively.}
\label{fig:cubemap}
\end{figure}

Note that if we always start with the top vertex of a cube then there are no merge carets, so that term drops out of $L_i$ and $R_i$.  Moreover, if we only look at maximal cubes then every component of $\Phi$ will be a caret.  Thus, if we start with top vertices of maximal cubes then $L_i \defeq i + \sum\limits_{j = 1}^{i-1} w_j$ and $R_i \defeq i + \sum\limits_{j = 1}^{i} w_j$. See Figure~\ref{fig:cubemap} for an example of such a situation.

\begin{lemma}\label{lem:C_well_def}
 The map $C$ is well defined.
\end{lemma}

\begin{proof}
 By definition $L_{i+1}-L_i\ge 1$ and $R_{i+1}-R_i\ge 1$, so these tuples really are in $\CF$. As a remark, if $w_i = 1$ and $c_i$ is a merge caret, or if $c_i$ is an edge, then $L_i = R_i$. Hence $C(q)$ may have repeated entries, but as an element of $\CF$, $C(q)$ is equivalent to the tuple with the duplicates deleted.

 We need to show that $C$ is independent of the choices of $K$ and $x$, which by Lemma~\ref{lem:param_cts} amounts to showing that $C$ is well defined on equivalence classes of generalized strand diagrams under $\sim_g$.  The only moves that are non-trivial to check are reduction, expansion and the relation $\sim_f$ that adds or deletes carets with weight $0$. From the point of view of $C$, a reduction or expansion amounts to replacing either a split caret by a merge caret, or a merge caret by a split caret, and then in either case changing the weight $w$ on said caret to $1-w$. It is evident that no such moves will change any of the $L_i$ or $R_i$.
 
 Next note that the $L_i$ and $R_i$ cannot tell the difference between some $c_j$ being an edge or being a split caret with weight $0$. This shows that adding or deleting a split caret with weight $0$ does not change any $L_i$ or $R_i$. Finally suppose $c_j$ is a merge caret and $w_j=0$, so in particular $R_j=L_j+1$. Replace $c_j$ by a pair of edges, as allowed by $\sim_f$, and let $(L'_1,R'_1,\dots,L'_{\ell+1},R'_{\ell+1})$ be the resulting tuple (note that we now have $\ell+1$ components). For any $i<j$ we have $L'_i=L_i$ and $R'_i=R_i$. Also, $L'_j=R'_j=L_j$ and $L'_{j+1}=R'_{j+1}=R_j$. Finally, for $i>j+1$ we have $L'_i=L_{i-1}$ and $R'_i=R_{i-1}$. We conclude that $(L_1,R_1,\dots,L_\ell,R_\ell) = (L'_1,R'_1,\dots,L'_{\ell+1},R'_{\ell+1})$ as elements of $\CF$.
\end{proof}

\begin{lemma}[From strands to configurations]\label{lem:embedding_works}
 The map $C \colon X \to \CF$ is continuous, constant on $F$-orbits and induces an injection $\overline{C} \colon \overline{X} \hookrightarrow \CF$. Moreover $\overline{C}$ is a homeomorphism onto the image $C(X)$.
\end{lemma}

\begin{proof}
 That $C$ is continuous follows from Lemmas~\ref{lem:param_cts} and~\ref{lem:C_well_def}. The action of $F$ on $X$ preserves the number of sinks and the weights of a generalized strand diagram (since $F$ acts on the left and generalized strand diagrams have multiplication by a weighted elementary forest on the right), so $C$ is constant on $F$-orbits. Hence we get a map $\overline{C} \colon \overline{X} \to \CF$.
 
 We next claim that $\overline{C}$ is injective. For each $n$ there is exactly one $F$-orbit in $\overline{X}$ of vertices with $n$ sinks, and said orbit maps to $(1,2,\dots,n)$ under $\overline{C}$, so $\overline{C}$ is visibly injective on vertex orbits in $\overline{X}$. Now suppose $q_1,q_2\in X$ are not vertices, and say $C(q_1)=C(q_2)$. Choose cubes $K_1,K_2$ such that $q_i$ lies in the interior of $K_i$ ($i=1,2$), and let $x_i$ be the top vertex of $K_i$. The tuple $C(q_i)$ has precisely twice as many distinct entries as $C(x_i)$, which tells us that $C(x_1)=C(x_2)$ and hence $F.x_1=F.x_2$. This then implies that $F.q_1=F.q_2$, proving the claim.
 
 The last thing to show is that $\overline{C}$ is a homeomorphism onto its image $C(X)$, and so it suffices to show that it is an open map. This follows from the definition (and continuity) of $\overline{C}$.
\end{proof}

The next step is to homotope $\CF$ to its subspace $C(X)$. Consider the subspace $\DF \subseteq \CF$ consisting of those points in $\CF$ that can be represented by tuples $(t_1,\dots,t_n)$ satisfying the additional requirements for all relevant $i$:

\begin{enumerate}
 \item $t_1=1$,
 \item $t_{i+1}-t_i \le 1$ and
 \item if $t_{i+1}-t_i < 1$ then $t_i - t_{i-1} = 1$ and $t_{i+2} - t_{i+1} = 1$.
\end{enumerate}

We claim $\DF=C(X)$. Indeed, $L_1=1$, and $R_i-L_i \le 1$ and $L_{i+1}-R_i=1$ for all $i$, where $(L_1,R_1,\dots,L_\ell,R_\ell)$ is any point in $C(X)$, so we have $C(X)\subseteq \DF$.  To see that $\DF \subseteq C(X)$, note that if $c_i$ is a split caret, then $R_i-L_i=w_i$, so it is easy to produce a generalized strand diagram mapping under $C$ to an arbitrary point of $\DF$.

\begin{proposition}\label{prop:retract_to_CF}
 The space $\CF$ is homotopy equivalent to its subspace $\DF$.
\end{proposition}

\begin{proof}
 We can homotope to $\DF$ in three steps. In the first step, we multiply all values by $2$.  Note that this ensures that $t_{i+1}-t_{i-1} \geq 2$ for all $i$, and so if $t_{i+1} - t_i < 1$ then both $t_{i+2} - t_{i+1} > 1$ and $t_{i} - t_{i-1} > 1$.  Next, we linearly translate $(t_1,\dots,t_n)$ until $t_1=1$, satisfying Condition (1) of $\DF$ above. In the third step, proceeding from left to right, any time there is an $i$ with $t_{i+1}-t_i>1$, we linearly translate all the $t_j$ for $j>i$ in sync until $t_{i+1}=t_i+1$.  The resulting tuple satisfies Conditions (1) and (2) above and, together with the note from the first step, also satisfies Condition (3).  It is easy to see that all these moves are indeed homotopy equivalences.
\end{proof}

We conclude that $\CF \simeq \DF \cong \overline{X}$ is a classifying space for $F$, proving Theorem~\ref{thm:CF}.

\begin{remark}\label{rmk:other_models}
 It is worth mentioning here some alternate configuration space models of classifying spaces for $F$ that could be useful. These were communicated to us by Bux, and are due to Belk, though have not appeared in print. We will not be overly detailed here.
 
 \begin{enumerate}
  \item First, we could homotope $\R$ to $(0,1)$ in an appropriate way and obtain a space $\CF'$ described as follows. Points in $\CF'_n$ are non-decreasing $n$-tuples $(t_1,\dots,t_n)$ with $t_i\in (0,1)$ satisfying the requirement that $t_{i+2}-t_i \ge 1/3n$ for all $1\le i \le n-2$. Then $\CF'$ is the union of the $\CF_n'$ modulo $\sim$ in the usual way. The advantage is that we can work with configurations in the bounded interval $(0,1)$ (or $[0,1]$ if we wish); the trade-off is that when we say no three points may be ``close,'' the definition of close depends on the total number of (distinct) points in the configuration.
 
  \item We can also consider configurations not of points, but of subintervals of the line. Said intervals have width say between $1$ and $2$ (inclusive), and a split amounts to replacing an interval of length exactly $2$ by its two halves, which are intervals of length $1$. This viewpoint is more in line with thinking of elements of $F$ not as strand diagrams but as homeomorphisms of the unit interval.
 
  \item Lastly we mention that if we use the unit circle instead of the unit interval, then, as mentioned by Belk in \cite{belk04}, we expect to obtain a classifying space for a cover of Thompson's group $T$. One could also consider configurations on other graphs, and obtain other generalizations of Thompson's group. Such things could also be viewed as generalizing \emph{graph braid groups}, the extra data being the ability of points to split and merge.
 \end{enumerate}
\end{remark}

As a remark, it would be interesting to exhibit an explicit homotopy equivalence from $\CF$ (or from $\overline{X}$) to a classifying space for $F$ with compact $n$-skeleton for each $n$.

We conclude by mentioning that, considering braid groups and $F$ admit classifying spaces given by configurations, one might hope that the braided Thompson's groups introduced by Brin \cite{brin07} and Dehornoy \cite{dehornoy06} also admit classifying spaces given by some sort of configurations. However, this appears to be significantly harder to prove than in the cases of braid groups and $F$.

\bibliographystyle{alpha}
\bibliography{Thompson_classifying_spaces}

\end{document}